\documentclass[12pt,reqno]{amsart}
\usepackage{amssymb}
\usepackage{amsmath}
\usepackage{mathrsfs}
\usepackage{amssymb, amsmath, amsfonts}

\usepackage{mathabx,url}

\usepackage{color}

\usepackage[hidelinks]{hyperref}

\makeatletter
\def\tank#1{\protected@xdef\@thanks{\@thanks
		\protect\footnotetext[0]{#1}}}
\def\bigfoot{
	
	\@footnotetext}
\makeatother

\newcommand{\ea}{\end{array}}

\allowdisplaybreaks
\numberwithin{equation}{section}
\newtheorem{theorem}{Theorem}[section]
\newtheorem{lemma}[theorem]{Lemma} 
\newtheorem{proposition}[theorem]{Proposition} 

\newtheorem{corollary}[theorem]{Corollary} 

\newtheorem{cor}[theorem]{Corollary}

\newtheorem{condition}[theorem]{Condition} 

\def\beq{\begin{equation}}
\def\nneq{\end{equation}}

\def\bthm{\begin{theorem}}
\def\nthm{\end{theorem}}

\def\blem{\begin{lemma}}
\def\nlem{\end{lemma}}
\def\bprf{\begin{proof}}
\def\nprf{\end{proof}}
\def\bprop{\begin{prop}}
\def\nprop{\end{prop}}
\def\brmk{\begin{rem}}
\def\nrmk{\end{rem}}

\def\bexa{\begin{exa}}
\def\nexa{\end{exa}}
\def\bcor{\begin{cor}}
\def\ncor{\end{cor}}

\newcommand\HH{\mathcal H}

\def\e{{\varepsilon}}

 \def\a{\boldsymbol {a}}
 \def\x{\boldsymbol {x}}
\def\y{\boldsymbol {y}}
\def\z{\boldsymbol {z}}
\def\h{\boldsymbol {h}}
\def\w{\boldsymbol {w}}

\title[Analysis of the gradient  for the stochastic  fractional heat equation]{Analysis of the gradient  for the stochastic  fractional heat equation with spatially-colored noise in $\mathbb R^d$}

\author[R. Wang]{Ran Wang}
\address[]{Ran Wang, School of Mathematics and Statistics,  Wuhan University,  Wuhan, 430072,
China.}
\email{rwang@whu.edu.cn}

\date{}
\begin{document}
\maketitle

 \noindent {\bf Abstract:}  Consider  the stochastic partial differential equation
  \begin{equation*}
    \frac{\partial }{\partial t}u_t(\x)= -(-\Delta)^{\frac{\alpha}{2}}u_t(\x) +b\left(u_t(\x)\right)+\sigma\left(u_t(\x)\right) \dot F(t, \x), \ \ \ t\ge0, \x\in \mathbb R^d,
    \end{equation*}
    where  $-(-\Delta)^{\frac{\alpha}{2}}$ denotes  the fractional Laplacian with the power $\alpha/2\in (1/2,1]$,  and the driving noise $\dot F$ is  a centered  Gaussian random field which is white in time and with a spatial homogeneous covariance given by the Riesz kernel. We study the detailed behavior of the approximation spatial gradient $u_t(\x)-u_t(\x-\e \boldsymbol  e)$ at any fixed time $t>0$, as $\e\downarrow 0$, where $\boldsymbol e$ is  {a unit vector} in $\mathbb R^d$. As applications, we deduce the law of iterated logarithm and  the behavior of the $q$-variations of the solution in space.

     \vskip0.3cm
 \noindent{\bf Keyword:} {Stochastic heat equation; Fractional Brownian motion; Fractional Laplacian; Gradient estimates.  }
 \vskip0.3cm

\noindent {\bf MSC: } {60H15; 60G17.}
\vskip0.3cm

     \section{Introduction}

    Consider the stochastic fractional heat equation
    \begin{equation}\label{eq SFBE}
    \frac{\partial }{\partial t}u_t(\x)= -(-\Delta)^{\frac{\alpha}{2}}u_t(\x) +b(u_t(\x))+\sigma(u_t(\x)) \dot F(t, \x), \ \ \ t\ge0, \x\in \mathbb R^d,
    \end{equation}
where $b$ and $\sigma$ are assumed to be  Lipschitz continuous functions, $\alpha\in (1,2]$ is a fixed ``spatial scaling" parameter, $-(-\Delta)^{\frac{\alpha}{2}}$ denotes  the fractional Laplacian with the power $\alpha/2$,  and the driving noise $\dot F$ is  a centered  Gaussian field which is white in time and with a spatial homogeneous covariance given by the Riesz kernel $f(x)=c_{1,1}\|x\|^{-(d-\gamma)}$ of the form:
 \begin{align}\label{eq noise}
 \mathbb E\left[\dot F(t,\x) \dot F(s,\y) \right]=c_{1,1}\delta_0(t-s)\|\x-\y\|^{-(d-\gamma)}, \ \ 0<\gamma<d,
 \end{align}
 where $\delta_0$ denotes the Dirac delta function and
  \begin{align}\label{eq c11}
  c_{1,1}=  2^{d-\gamma}\pi^{\frac{d}{2}}\Gamma\left((d-\gamma)/2\right) /\Gamma\left( \gamma/2\right).
 \end{align}

Throughout this paper, we assume that
    \begin{condition}\label{cond u0}
    \begin{itemize}
     \item[(a)] $\alpha\in (1,2]$, $d\ge1$ and $\gamma\in ((d-\alpha)_+, d)$.
    \item[(b)]Let $u_0=\{u_0(\x)\}_{\x\in \mathbb R^d}$ be a random field. Assume that there exist real numbers $k_0>\max\{2,\, 2/(2+d-\alpha-\gamma)\}$, $\eta_0\in (\frac{\alpha-d+\gamma}{\alpha},1]$  and $c_{1,2}>0$ such that
    \begin{align}\label{eq u0 1}
    \sup_{\x\in \mathbb R^d}\mathbb E\left[ |u_0(\x)|^{k_0}\right]<\infty
    \end{align}
    and
    \begin{align}\label{eq u0 2}
     \mathbb E\left[|u_0(\x)-u_0(\x+\h)|^{k_0} \right]\le c_{1,2}\|\h\|^{\eta_0 k_0},
    \end{align}
    uniformly for all $\x, \h\in \mathbb R^d$.
  \item[(c)] The functions  $b$ and $\sigma:  \mathbb R \rightarrow \mathbb R$ are   Lipschitz continuous, that is,
        \begin{align}\label{eq Lip}
    |b(r)-b(v)| \le L_{b}|r-v|, \,\,\,    |\sigma(r)-\sigma(v)|\le L_{\sigma}|r-v|  \ \ \ \text{for } r,v\in \mathbb R,
       \end{align}
       where $L_b$ and $L_{\sigma}$ are some finite positive constants.
       \end{itemize}
          \end{condition}

          Under  Condition \ref{cond u0},    by the theory of Dalang \cite{Dalang1999}, there exists a unique continuous solution of   \eqref{eq SFBE}, satisfying  that for any $T>0,k\in[2,k_0]$,
       \begin{align}\label{eq u moment}
    \sup_{(t,\x)\in [0,T]\times \mathbb R^d}\mathbb E\left[|u_t(\x)|^{k} \right]<+\infty.
       \end{align}
        See, for instance, \cite{ANTV2022} and references therein for details.

        When $d=1$ and $F$ is the space-time white noise, Foondun et al. \cite{FKM2015} utilized an approximation approximation approach to study the local and variational properties of the spatial processes $\{u_t(\x)\}_{\x\in \mathbb R}$, where $t>0$ is fixed. Their key idea in \cite{FKM2015} is to show that, as $\varepsilon \downarrow0$,
        $$
        u_t(\x)-u_{t}(\x-\e)\approx c_{1,3}\sigma(u_t(\x)) \left[B^{H_0}(\x)-B^{H_0}(\x-\e) \right] \ \ \   \text{in certain sense},
        $$
 where $c_{1,3}= (2\Gamma(\alpha)|\cos(\alpha\pi/2)|)^{-1/2}$, and $B^{H_0}$ denotes a  fractional Brownian motion (fBm, for short) with Hurst index
        $H_0= (\alpha-1)/2\in (0,1/2]$.    They were able to quantify the size of the approximation error by controlling its moments. Consequently, some of the local properties of
        $\x\mapsto u_t(\x)$  can be derived from those of fBm  $B^{H_0}$.

        In this paper, we consider the stochastic fractional heat equation (1.1) with the time-white and space-colored Gaussian noise and extend the approximation approach in \cite{FKM2015} to the more general setting.  In this case,  the solution will be related to the isotropic multiparameter   fractional Brownian motion     (also known as the L\'evy fBm) $\{B^H(\x)\}_{\x\in \mathbb R^d}$, which   is defined  as a centered Gaussian process, starting from zero, with covariance function
     \begin{align}\label{eq Isot fBM}
     \mathbb E\left[B^H(\x)B^H(\y) \right]=\frac12\left(\|\x\|^{2H}+\|\y\|^{2H}-\|\x-\y\|^{2H}\right)   \ \  \ \  (\x, \y\in \mathbb R^d).
     \end{align}
      {
    Here,  the Hurst index is  $H=(\alpha-d+\gamma)/{2} \in (0, 1)$,   which will be used throughout the whole of this paper.}

     Our main result is  the following.
    \begin{theorem}\label{thm spatial gradient} Under Condition \ref{cond u0},
    for every fixed $t>0$, there exists   an isotropic multiparameter fBm $\{B^{H}(\x)\}_{\x\in \mathbb R^d}$,    with Hurst index $H=(\alpha-d+\gamma)/2$     such that for all $\lambda>0$,
    \begin{align}\label{eq spatial gradient }
    \lim_{\varepsilon\downarrow0}\sup_{\x\in \mathbb R^d}\sup_{\|\boldsymbol {e}\|=1  }\mathbb P\left(\left| \frac{u_t(\x)-u_t(\x-\varepsilon \boldsymbol {e})}{B^{H}(\x)-B^{H}(\x-\varepsilon \boldsymbol {e})}-  c_{\alpha, \gamma,d} \sigma(u_t(\x))   \right| >\lambda\right)=0,
\end{align}
where  $c_{\alpha, \gamma,d}$ is the following numerical constant:
  \begin{align}\label{eq constant c}
   c_{\alpha, \gamma,d}=(2\pi)^{-d/2}\left(\int_{\mathbb R^d} \|\boldsymbol  w\|^{-(\alpha+\gamma)}\big(1-\cos(\boldsymbol w\cdot \boldsymbol e)\big)d\boldsymbol  w\right)^{1/2},
   \end{align}
with  $\boldsymbol e$ being a unit vector in $\mathbb R^d$.
 \end{theorem}

       By using the arguments in the proofs of Corollaries 1.2, 1.3 and 1.5    in \cite{FKM2015},  we can obtain the following results.
            \begin{corollary}\label{cor  t LIL} Assume that Condition \ref{cond u0} holds.
    Choose and fix $t>0$, $\x\in \mathbb R^d$ and a  unit vector $\boldsymbol  e\in \mathbb R^d$. Then with probability one,
    \begin{equation}
    \begin{split}
    \limsup_{\e\downarrow 0}\frac{u_t(\x)-u_t(\x-\e\boldsymbol  e)}{\e^{H}\sqrt{2 \log\log(1/\e)}}= c_{\alpha, \gamma,d}|\sigma(u_t(\x))|.
    \end{split}
    \end{equation}
         \end{corollary}

        \begin{corollary} Assume that Condition \ref{cond u0} holds.
    Choose and fix $t>0$, $\x\in \mathbb R^d$ and a unit vector $\boldsymbol e\in \mathbb R^d$. Then for all $\theta\in \mathbb R$,
        \begin{equation}
        \lim_{\e\downarrow 0}\mathbb P\left(\frac{u_t(\x)-u_t(\x-\e\boldsymbol e)}{ \e^{H}}\le \theta\right)
        =\mathbb P\left(\sigma(u_t(\x))\times \mathcal N\le \frac{\theta}{c_{\alpha, \gamma,d}} \right),
       \end{equation}
        where $\mathcal N$ denotes a standard Gaussian random variable, independent of $u_t(\x)$.
                                      \end{corollary}

                   Following the one-parameter case, we define the $q$-variation  {over an interval $[A_1,A_2]$} of the $d$-dimensional random field $\{\xi(\x)\}_{\x\in \mathbb R^d}$ as the limit in probability, as $n\rightarrow\infty$, of the sequence
                   \begin{align}\label{eq variation}
                   V_{[A_1, A_2]}^{n, q}(\xi):=\sum_{i=0}^{n-1}\left|\xi({\x}_{i+1})-\xi({\x}_{i}) \right|^q,
                   \end{align}
               where ${\x}_i=\left(x_i^{(1)}, \cdots, x_i^{(d)} \right)$ with $x_i^{(j)}=A_1+\frac{i}{n}(A_2-A_1)$ for $i=0,1, \cdots, n$ and $j=1,\cdots, d$.
                   See, e.g., \cite{KT2019}.

                    For any $a\in \mathbb R$, denote $\a:=(a, \cdots, a)\in \mathbb R^d$.

        \begin{corollary}\label{prop variations}      Assume that Condition \ref{cond u0} holds.        If $\varphi:\mathbb R\rightarrow\mathbb R$ is Lipschitz continuous,  then for all   non random reals $A_2>A_1$ and   $t>0$,
        \begin{equation}\label{eq variations}
        \begin{split}
        &\lim_{n\rightarrow\infty} \sum_{A_1 2^n\le j \le A_2 2^n}\varphi(u_t({\x}_{j/2^n}))\left|u_t({\x}_{(j+1)/2^n})-u_t({\x}_{j/2^n}) \right|^{1/H}\\
        =& \,    c_{1,4}  \sqrt d   \int_{A_1}^{A_2}\varphi(u_t({\a})) \sigma(u_t({\a}))^{1/H}da,
        \end{split}
        \end{equation}
        almost surely and in $L^2(\Omega, \mathbb P)$, where $c_{1,4}=c^{1/H}_{\alpha, d, \gamma} \mathbb E|\mathcal N|^{1/H}$ with   $\mathcal N$ being a standard Gaussian random variable.         Particularly, taking $\varphi\equiv1$, we get the $1/H$-variation  of $\{u_t(\x)\}_{\x\in \mathbb R^d}$:        \begin{equation}\label{eq variations}
        \begin{split}
           \lim_{n\rightarrow\infty}V_{[A_1, A_2]}^{2^n, {1/H}}(u_t)
        =  \,    c_{1,4}    \sqrt d   \int_{A_1}^{A_2}\sigma(u_t({\a}))^{1/H}da,
        \end{split}
        \end{equation}
        almost surely and in $L^2(\Omega, \mathbb P)$.
        \end{corollary}
 {
         At the end of the section,   we   briefly introduce some works about   the approximation temporal gradient $u_{t+\e}(\x)-u_t(\x)$ at any fixed  $t > 0$ and $\x\in \mathbb R^d$, as $\e\downarrow 0$.  When $d=1$ and $F$ is the space-time white noise, Khoshnevisan et al. \cite{KSXZ2013}
        utilized an approximation approach to study the local and variational  properties of the temporal process
        $\{u_t(\x)\}_{t \ge 0}$, where $\x\in \mathbb R^d$ is fixed.
        Their key idea in \cite{KSXZ2013} is to show that, as $\e\downarrow 0$,
                $$
        u_{t+\e}(\x)-u_{t}(\x)\approx c_{1,5}\sigma(u_t(\x)) \left[B^{\widetilde H}(t+\e)-B^{\widetilde H}(t) \right] \  \ \text{in certain sense},
        $$
        where $c_{1,5}=\frac{1}{\pi(\alpha-1)}\Gamma\left(\frac{1}{\alpha}\right)^{\frac12}$,
        and $B^{\widetilde H}$ denotes an fBm with Hurst index $\widetilde H=(\alpha-1)/(2\alpha)\in (0,1/4]$.
        Das \cite{Das2022} further applied this approximation approach and studied the sample path properties of the temporal
        process of the Kardar-Parisi-Zhang equation with general initial data.      Recently,  Wang and Xiao \cite{WX2023} consider the stochastic fractional heat equation (\ref{eq SFBE}) with the time-white and space-colored   Gaussian noise and extend the approximation approach in \cite{KSXZ2013} to the more general setting.}

      The rest of this paper is organized as follows. In Section 2, we first  introduce the stochastic integral and   give some facts about the linear stochastic heat equation  taking from \cite{KT2019}, then we prove the H\"older continuity of the solution.  In Section 3, we give some estimates of the localization of the solution.   Section 4 is  devoted to the  proof of the main results in this paper.

          \section{Preliminaries}
  \subsection{Stochastic integral}
       We first define precisely the driving noise that appears in \eqref{eq SFBE}, which is borrowed  from \cite{DKN2013}.
   Let $\mathcal D(\mathbb R^d)$ be the space of $C^{\infty}$-test functions with compact support. Then $F=\{F(\phi), \, \phi\in \mathcal D(\mathbb R^{d+1})\}$ is an $L^2(\Omega, \mathcal F, \mathbb P)$-valued mean zero Gaussian process with covariance
\begin{align}\label{eq F1}
 \mathbb E\left[F(\phi)F(\psi) \right]=c_{1,1}\int_{\mathbb R_+}dr \int_{\mathbb R^d}d\y \int_{\mathbb R^d}d\z \,\phi(r,\y)\|\y-\z\|^{-(d-\gamma)}\psi(r,\z).
  \end{align}
 Using elementary properties of the Fourier transform (see \cite{Dalang1999}), this covariance can also be written as
\begin{equation}\label{eq F2}
  \mathbb E\left[F(\phi)F(\psi) \right]=   \int_{\mathbb R_+}d r\int_{\mathbb R^d}d\xi\, \|\xi\|^{-\gamma}\mathcal F\phi(r,\cdot)(\xi)\overline{\mathcal F\psi(r,\cdot)}(\xi),
  \end{equation}
  where   $\mathcal F f (\cdot)(\xi)$ denotes the Fourier transform of $f$, that is,
  $$
  \mathcal F f(\cdot) (\xi)=\int_{\mathbb R^d} e^{2\pi i \xi\cdot \x} f(\x)d\x.
  $$

  Following Walsh \cite{Walsh1986} and Dalang \cite{Dalang1999}, a rigorous formulation of  \eqref{eq SFBE} through the notion of mild solution as follows. Let $M=\{M_t(A), \, t\ge0, A\in \mathcal B_b(\mathbb R^d)\}$ be the worthy martingale measure obtained as an extension of the process $\dot F$ as in Dalang and Frangos  \cite{DF1998} (also see Dalang and Quer-Sardanyons \cite{DQ2011}).   Then a mild solution of \eqref{eq SFBE} is jointly measurable $\mathbb R$-valued process   $u=\{u(t,\x)\}_{t\ge0, \x\in \mathbb R^d}$,  adapted to the natural filtration generated by $M$, such that
  \begin{equation}\label{eq solution u}
\begin{split}
u_t(\x)=& \int_{\mathbb R^d}G^{\alpha}_{t}(\x, \y) u_0(\y)d\y+ \int_0^t\int_{\mathbb R^d}G^{\alpha}_{t-s}(\x, \y) b(u_s(\y))dsd\y\\
&\,+\int_0^t\int_{\mathbb R^d}G^{\alpha}_{t-s}(\x, \y) \sigma(u_s(\y)) M(ds,d\y),
\end{split}
\end{equation}
 where $G^{\alpha}_t(\x,\y):=G^{\alpha}_t(\x-\y)$ is the    Green kernel associated to the operator $-(-\Delta)^{\frac{\alpha}{2}}$ on $\mathbb R^d$, which is defined via its Fourier transform
 \begin{align}\label{eq heat kernel}
 (\mathcal FG^{\alpha}_t)(\cdot)(\xi)=e^{-t\|\xi\|^{\alpha}}, \ \ \ \xi\in \mathbb R^d,
 \end{align}
 for $\alpha\in (1,2]$, and  the stochastic integral is interpreted in the sense of Walsh \cite{Walsh1986}.
 We note that the covariance measure $M$ is
 $$
 Q([0,t]\times A\times B):=\langle M(A), M(B)\rangle_t=t\int_{\mathbb R^d}d\x \int_{\mathbb R^d}d\y\, 1_{A}(\x)\|\x-\y\|^{-(d-\gamma)}1_{B}(\y)
 $$
 and its dominating measure $K\equiv Q$.        In particular,  by  \eqref{eq F1},   \eqref{eq F2} and  \eqref{eq heat kernel}, we have
  \begin{equation}\label{eq kernel 1}
 \begin{split}
& \mathbb E\left[\left(\int_0^t\int_{\mathbb R^d}G^{\alpha}_{t-s}(\x, \y)  M(ds,d\y)\right)^2 \right]\\
 =&\, c_{1,1}\int_0^tds\int_{\mathbb R^d}d\y\int_{\mathbb R^d}d\z\, G^{\alpha}_{t-s}(\x, \y) \|\y-\z\|^{-(d-\gamma)} G^{\alpha}_{t-s}(\x, \z) \\
  =&\,  \int_0^tds\int_{\mathbb R^d}d\xi\, \|\xi\|^{-\gamma}\, |\mathcal F G^{\alpha}_{t-s}(\cdot)(\xi)|^2\\
  = &\,  \int_0^tds\int_{\mathbb R^d}d\xi\, \|\xi\|^{-\gamma}\, e^{-2(t-s)\|\xi\|^{\alpha}}\\
    =&\, c_{2,1}  \int_0^t s^{-(d-\gamma)/\alpha} ds,
     \end{split}
 \end{equation}
 where  $c_{2,1}=\int_{\mathbb R^d} \|\xi\|^{-\gamma}\, e^{-2 \|\xi\|^{\alpha}}d\xi<\infty$. The integral $ds$ in  the last term of \eqref{eq kernel 1} is finite if and only if $d-\gamma<\alpha$.

For any random variable $\zeta\in L^p(\Omega)$ with $p\ge1$, let $\|\zeta\|_{L^p(\Omega)}:=\left(\mathbb E |\zeta|^p\right)^{\frac{1}{p}}$.
 For any $\phi, \psi\in \mathcal D(\mathbb R^d)$, let
\begin{equation}\label{eq F3}
\begin{split}
 \langle \phi, \psi \rangle_{\HH}:=&\,   \int_{\mathbb R^d}d\xi\, \|\xi\|^{-\gamma}\mathcal F\phi(\cdot)(\xi)\overline{\mathcal F\psi(\cdot)}(\xi)\\
 =&\, c_{1,1} \int_{\mathbb R^d}d\y\int_{\mathbb R^d}   d\z\, \phi(\y) \|\y-\z\|^{-(d-\gamma)} \psi(\z).
 \end{split}
  \end{equation}
 Denote by $\HH$ the Hilbert space obtained by the completion of $\mathcal D(\mathbb R^d)$ with respect to the inner product $\langle \phi, \, \psi\rangle_{\mathcal H}$ defined by \eqref{eq F3}.

 According to the proof of Lemma 5.4 in \cite{FK2013} or Proposition 2.1 in \cite{SW2020}, we have the following Burkholder-Davis-Gundy inequality for the   stochastic convolution driven by the time-white and  space-colored noise.
 \begin{proposition}$($\cite[Lemma 5.4]{FK2013},   \cite[Proposition 2.1]{SW2020}$)$\label{prop BDG}
  Let $\{\sigma(t,\x)\}_{(t,\x)\in[0,T]\times \mathbb R^d}$ be a predictable random field  such that the following stochastic integral is well-defined.
    Then for any $t\in[0,T], \x\in \mathbb R^d$ and $p\geq 2$,
  \begin{equation}\label{eq BDG}
        \left\|  \int_0^t\int_{\mathbb R^d} G^{\alpha}_{t-s}(\x,\y)\sigma(s,\y)M(ds,d\y) \right\|_{L^p(\Omega)}^2
    \leq   4p \int_0^t \left\|G^{\alpha}_{t-s}(\x,\cdot) \|\sigma(s, \cdot)\|_{L^p(\Omega)}\right\|_{\HH}^2ds.
  \end{equation}
\end{proposition}

       \subsection{The fractional  Green kernel}
    Let us recall some useful properties of the kernel $G^{\alpha}_t(\x,\y)=G^{\alpha}_t(\x-\y)$  defined through \eqref{eq heat kernel}.   For details, we refer to \cite{ANTV2022, CZ2016}.  
    
    It is well known that   $G^{\alpha}_t(\cdot)$ is the probability transition density function of a rotationally invariant   $d$-dimensional  stable process $\{L_t^{(\alpha)}\}_{t\ge0}$.   By the scaling property of $L_t^{(\alpha)}\overset{d}{=}t^{1/\alpha} L_1^{(\alpha)}$, it is easy to see that
   \begin{equation}\label{eq scaling}
    G^{\alpha}_t(\x)=t^{-\frac{d}{\alpha}} G^{\alpha}_1\left(t^{-\frac{1}{\alpha}}\x\right) \ \ \ \ (t>0,\, \x\in \mathbb R^d).
  \end{equation}

  When $\alpha=2$, $L_t^{(\alpha)}$  is the standard Brownian motion, and
$$G^{\alpha}_t(\x,\y)=\frac{1}{(2\pi t)^{d/2}}\exp\left\{-\frac{\|\x-\y\|^2}{2t}\right\}.$$
  { For simplicity of presentation, we assume that $\alpha\in(1,2)$ throughout the rest of this paper. The proof in the case of $\alpha= 2$ is relatively simple, which is omitted here.}

 When $\alpha\in (1,2)$,  by \cite[Theorem 2.1]{BG1960}, there exist   some  finite constants  $0<K_{\alpha}'< K_{\alpha}<\infty$  such that for all $t>0,\x,\y\in \mathbb R^d$,
\begin{align}\label{eq Green 3}
K_{\alpha}'\frac{ t}{\left(t^{1/\alpha}+\|\x-\y\|\right)^{d+\alpha}} \le G^{\alpha}_t(\x,\y)\le K_{\alpha}\frac{ t}{\left(t^{1/\alpha}+\|\x-\y\|\right)^{d+\alpha}}.
   \end{align}
       By \cite[Lemma 2.2]{CZ2016}, one has that for every $T>0$, there exists a constant $c_{2,2}>0$ such that for all $0< t\le T$ and $\x, \y,\z\in \mathbb R^d$,
 \begin{align}\label{eq Green 5}
       \left| G^{\alpha}_t(\x,\y)-G^{\alpha}_t(\x,\z)\right|\le  c_{2,2} \left(  \frac{\|\y-\z\| }{t^{1/\alpha}} \wedge 1\right) \times \big(G^{\alpha}_t(\x,\y)+ G^{\alpha}_t(\x, \z)\big).
      \end{align}

     \subsection{Linearization of  the  stochastic heat equation}

 Let us consider the following linearization of the stochastic heat equation:
 \begin{align}\label{eq linear SHE}
 \frac{\partial}{\partial t}Z_t(\x)=-(-\Delta)^{\frac{\alpha}{2}}Z_t(\x)+\dot F(t,\x),
 \end{align}
subject to $Z_0(\x)=0$ for all $\x\in \mathbb R^d$.   The solution of \eqref{eq linear SHE} can be written as
\begin{align}\label{eq mild SHE}
  Z_t(\x)=\int_0^t\int_{\mathbb R^d} G^{\alpha}_{t-s}(\x, \y) M(ds,d\y),  \ \ \ \ t>0, \x\in \mathbb R^d.
     \end{align}
 It is well-known  that the mild solution \eqref{eq mild SHE} is well-defined if and only if $d<\alpha+\gamma$. Moreover, in this case, for every $T>0, p\ge 2$,
     $$
     \sup_{t\in [0,T],\x\in \mathbb R^d}\mathbb E\left[|Z_t(\x) |^p\right]<+\infty.
     $$
     See, e.g., Proposition 4.1 in \cite{KT2019} or Remark 5.4 in \cite{Balan12}.

     The following bounds are well-known; see, e.g., Propositions 3.2 and 3.3 in \cite{KS2022} or references therein.
                     \begin{lemma}\label{lem Z spatial}$($\cite[Propositions 3.2 and 3.3]{KS2022}$)$ Fix $T>0$ and $k\ge2$.
             \begin{itemize}
             \item[(a)]        There exists a constant $c_{2,3}>0$ such that
     \begin{align}\label{eq Z10}
      \left\|Z_t(\x)-Z_t(\x-\h ) \right\|_{L^k(\Omega)} \le  c_{2,3} |\h|^{ H},
          \end{align}
     uniformly for all     $t\in [0,T]$ and $\x, \h\in \mathbb R^d$.
                \item[(b)]        There exists a constant $c_{2,4}>0$ such that
     \begin{align}\label{eq Z20}
     \left\|Z_t(\x)-Z_s(\x)\right\|_{L^k(\Omega)} \le  c_{2,4} |t-s|^{ H/\alpha},
          \end{align}
     uniformly for all    $t,s\in [0,T]$ and $\x  \in \mathbb R^d$.
          \end{itemize}
           \end{lemma}

       \begin{lemma}  {$($\cite[Proposition 4.6]{KT2019}$)$}\label{Lem  spatial}
 Fix $t>0$. Then the process $\{Z_t(\x)\}_{\x\in \mathbb R^d}$ defined by  \eqref{eq mild SHE} has the same finite-dimensional distribution as
 $$
 c_{\alpha, \gamma,d} B^{H}(\x)+S_t(\x)   \ \ \ \  (\x\in \mathbb R^d),
 $$
  where $B^{H}$ is an isotropic multiparameter fBm with Hurst index $H=(\alpha-d+\gamma)/2$, $c_{\alpha, \gamma,d}$ is the constant defined by \eqref{eq constant c},
  and   $\{S_t(\x)\}_{\x\in \mathbb R^d}$ is a Gaussian process with $C^{\infty}$-paths. Moreover, for any   $0<T_0<T_1<\infty$, there exists a constant $c_{2,5}:=c_{2,5}(T_0,T_1)>0$  such that
  $$
 \left\|S_t(\x)-S_t(\y)\right\|_{L^2(\Omega)}\le c_{2,5}\|\x-\y\|\ \ \ \ (\x, \y\in \mathbb R^d),
  $$
uniformly for  all  $t\in [T_0,T_1]$.
   \end{lemma}

 {By using the idea of pinned string process of Mueller and Tribe \cite{MT02}, it can be shown (cf. \cite{TX17, KT2019, HSWX20} in increasing generality) that the following decomposition result holds.}
     \begin{lemma}\label{lem Z temp}  For any fixed $\x\in \mathbb R^d$,  the temporal process $\{Z_t(\x)\}_{t\ge0}$
 has the same  distribution  as
\begin{equation}\label{eq decom}
c_{\alpha, \gamma,d} B^{H/\alpha}(t)+S(t)   \ \ \ \  (t\in \mathbb R_+ := [0, \infty)),
\end{equation}
where $B^{H/\alpha}$ is an  fBm with Hurst index $H/\alpha=(\alpha-d+\gamma)/{(2\alpha)}$,
   \begin{align}\label{eq c24}
   c_{2,6}=\left[  \frac{\alpha}{ (2\pi)^{d}2^{(d-\gamma)/\alpha}(\alpha-d+\gamma)}\int_{\mathbb R^d}
   \,\|\xi\|^{-\gamma}e^{-\|\xi\|^{\alpha}}d\xi \right]^{1/2},
    \end{align}
 and $\{S(t)\}_{t\in \mathbb R_+}$ is a Gaussian process with $C^{\infty}(0, \infty)$-paths.  Moreover,  for any
 $0<T_1<T_2<\infty$, there exists a constant $c_{2,7}:=c_{2,7}(T_0,T_1)>0$  such that
  \begin{align}\label{eq S}
  \left\|S(t)-S(s)\right\|_{L^2(\Omega)} \le c_{2,7}|t-s|,  \ \ \text{for   } s, t\in [T_0,T_1].
  \end{align}
          \end{lemma}

           By    the properties of the (isotropic multiparameter) fBm and Lemmas  \ref{Lem  spatial}  and \ref{lem Z temp} , it is easily to obtain  the following results, which will be   very important in our later needs.
            \begin{corollary}\label{coro Z spatial eq}
             \begin{itemize}
             \item[(a)]      Fix $T_1>T_0>0$.     Then,
     \begin{align}\label{eq Z4}
    \left\|Z_t(\x)-Z_t(\x-\e \boldsymbol  e )\right\|_{L^2(\Omega)} = c_{\alpha, \gamma, d} \e^{ H} +O\left(\e \right) \ \ \ \ (\varepsilon \downarrow0),
     \end{align}
     uniformly for  every $t\in [T_0,T_1]$,  $\x\in \mathbb R^d$   and every unit vector $\boldsymbol e$ in $\mathbb R$, where $ c_{\alpha, \gamma, d}$ is defined by  \eqref{eq constant c}.
     \item[(b)]
     Fix $T_1>T_0>0$. Then
     \begin{align}\label{eq Z1}
    \left\|Z_t(\x)-Z_s(\x )\right\| = c_{2,6} |t-s|^{H/\alpha} + {O\left(|t-s|\right)\ \ \ \ (|t-s|\downarrow0)},
     \end{align}
     uniformly for every $t,s\in [T_0,T_1]$ and  $\x\in \mathbb R^d$, $ c_{2,6}$ is defined by  \eqref{eq c24}.
     \end{itemize}
           \end{corollary}

 \subsection{H\"older continuity of the solution}

  \begin{proposition}\label{prop Holder}
 Under Condition \ref{cond u0},   for any $k\in [2,k_0]$ and $T>0$, there exists a constant $c_{2,8}>0$ such that
    \begin{align}\label{eq u xt}
    \left\|u_{t+\e}(\x)-u_{t}(\x+\h)\right\|_{L^k(\Omega)}   \le c_{2,8}\left(\e^{H/\alpha} +\|\h\|^{H}\right),
    \end{align}
    uniformly for all $t\in[0,T], \e\in (0,1),  \x, \h\in \mathbb R^d$.
       \end{proposition}

Next, we   decompose the solution $u$   in the following three terms:
\begin{equation}\label{eq decomposition}
u_t(\x)=\xi_t(\x)+X_t(\x)+Y_t(\x),
\end{equation}
  where
     \begin{align}
  \xi_t(\x):=&\, \int_{\mathbb R^d}G_{t}^{\alpha}(\x, \y)u_0(\y) d\y;\label{eq X0}\\
  X_t(\x):=&\,\int_0^t\int_{\mathbb R^d}G_{t-s}^{\alpha}(\x,\y)\sigma(u_s(\y)) M(ds,d\y);\label{eq X}\\
  Y_t(\x):= &\, \int_0^t\int_{\mathbb R^d}  G_{t-s}^{\alpha}(\x, \y) b(u_s(\y))   dsd\y.   \label{eq Y}
    \end{align}

        \begin{proof}[Proof of Proposition \ref{prop Holder}]
        By \eqref{eq u moment}, we only need to prove \eqref{eq u xt} for any $t\in[0,T], \e\in (0,1),  \x, \h\in \mathbb R^d$ with $\|\h\|\le 1$.
        
    For the first term $\{\xi_t(\x)\}_{t\in [0,T],\, \x\in\mathbb R^d}$,   by \cite[Proposition 2.6]{KS2022}  {and Condition \ref{cond u0}(b)},  there exist some  constants  $c_{2,9}>0$ and $\eta_0\in (\frac{\alpha-d+\gamma}{\alpha},1]$  such that
    \begin{align}\label{eq u0 3}
    \left\|\xi_{t+\e}(\x+\h) -\xi_t(\x) \right\|_{L^k(\Omega)}  \le c_{2,9}\left(\e^{\frac{\eta_0}2}+\|\h\|^{\eta_0}\right) \le   c_{2,9}\left(\e^{H/\alpha} +\|\h\|^{H}\right),
    \end{align}
    uniformly for all $t\in [0,T], \e\in (0,1)$ and $\x,\h\in \mathbb R^d$.   

    The second term  $X=\{X_t(\x)\}_{t\in [0,T], \x\in\mathbb R^d}$ solves the following equation:
         \begin{equation}\label{eq X 2}
         \frac{\partial }{\partial t}X_t(\x)=-(-\Delta)^{-\alpha/2} X_t(\x)+\sigma( u_t(\x))\dot F(t,\x), \ \ \ (t,\x)\in [0,T]\times \mathbb R^d,
         \end{equation}
    with vanishing initial condition $X_0(\x)=0$ for every $\x\in \mathbb R^d$.

      By   the Burkholder-Davis-Gundy inequality \eqref{eq BDG}   and  \eqref{eq kernel 1}, we have
       \begin{equation}\label{eq Holder X x}
       \begin{split}
      &\left\|X_{t}(\x+\h)-X_t(\x)\right\|_k^2\\
      = &\, \left\| \int_0^t\int_{\mathbb R^d} G_{t-s}^{\alpha}(\x, \y)  \left[\sigma(u_s(\y-\h))-\sigma(u_s(\y))\right] M(ds,d\y)    \right\|_k^2\\
      \le   &\,    4k \int_0^t \left\| G^{\alpha}_{t-s}(\x,\cdot)  \|\sigma(u(s,  \cdot-\h))-\sigma(u(s, \cdot))\|_{L^k(\Omega)}\right\|_{\HH}^2ds\\
    \le  &\,   4k L_{\sigma}^2\int_0^t  \sup_{\w\in \mathbb R^d} \|u(s,  \w+\h)-u(s, \w)\|_{L^k(\Omega)}^2\cdot\left\| G^{\alpha}_{t-s}(\x,\cdot) \right\|_{\HH}^2ds \\
      \le  &\, 4kc_{2,1} L_{\sigma}^2  \int_0^t  \, \sup_{\w\in \mathbb R^d} \|u(s,   \w+\h)-u(s, \w)\|_{L^k(\Omega)}^2 s^{-(d-\gamma)/\alpha}  ds,
      \end{split}
    \end{equation}
  uniformly for all    $t\in [0,T]$   and  $\x, \h\in \mathbb R^d$.

    Since $b$ is  global Lipschitz continuous, applying Minkowski's inequality,  there exists a  constant $c_{2,10}>0$ such that
      \begin{equation}\label{eq Holder Y x}
          \left\|Y_t(\x+\h)-Y_t(\x)\right\|_{L^k(\Omega)}^2\le c_{2,10}\int_0^t \sup_{\w\in \mathbb R^d}\|u(s,  \w+\h)-u(s,\w)\|_{L^k(\Omega)}^2 ds,
     \end{equation}
  uniformly for all $t\in [0,T], \x, \h\in \mathbb R^d$.

    Set
    $$
  \Lambda_{\h}(t):= \sup_{\w\in \mathbb R^d}\|u(t, \w+\h)-u(t,\w)\|_{L^k(\Omega)}^2, \ \ \ \ 0\le t\le T.
    $$
  By   \eqref{eq u0 3}, \eqref{eq Holder X x} and \eqref{eq Holder Y x}, there exists a constant $c_{2,11}>0$ such that
   $$
     \Lambda_{\h}(t)\le c_{2,11}\left(\h^{2H}+ \int_0^t \left( 1+s^{-(d-\gamma)/\alpha}\right)  \Lambda_{\h}(s)ds \right),
   $$
    uniformly in $t\in [0,T]$ and  $\h\in\mathbb R^d$. By the  fractional  Gr\"onwall's lemma, we have that for any $t\in [0,T]$ and $\|\h\|\le 1$,
        \begin{align}\label{eq Holder x}
    \Lambda_{\h}(t) \le  c_{2,12} (T)  \|\h\|^{2H},
    \end{align}
     where $c_{2,12} (T) \in (0,\infty)$.

     Similarly, by  \cite[Proposition 2.6]{KS2022} and  Lemma \ref{lem Z spatial}(b),  there exists a constant $c_{2,13}(T)>0$ such that
      \begin{align}\label{eq Holder t}
     \|u(t+\e, \x)-u(t,\x)\|_{L^k(\Omega)}\le c_{2,13}(T) \e^{ H/\alpha} ,
    \end{align}
    uniformly for all $t\in[0,T], \e\in (0,1)$ and $\x\in \mathbb R^d$.

    Putting \eqref{eq Holder x} and \eqref{eq Holder t} together, we obtain the desired result \eqref{eq u xt}.       The proof is complete.     \end{proof}

 The following estimate tells us that $\{Y_{t}(\x)\}_{t\ge0}$ defined by \eqref{eq Y}   is more regular.

    \begin{lemma}\label{Holder Y}For every $  \x, \y\in \mathbb R^d$ and for every $k\in [2,k_0]$, there exists a constant $c_{2,14}>0$ such that
   \begin{align}\label{eq regul Y}
    \sup_{t\in[0,T]}\mathbb E\left[ \left|Y_{t}(\x)-Y_t(\y) \right|^k\right]  \le c_{2,14}\|\x-\y\|^{k-1}.
  \end{align}
  \end{lemma}
  \begin{proof} For any $t\in [0,T]$ and $\x, \y\in \mathbb R^d$, by H\"older's inequality
    and  \eqref{eq Green 5}, we have that for any $k\in [2,k_0]$,
     \begin{equation*}
  \begin{split}
   \mathbb E\left[\left|Y_t(\x)-Y_t(\y) \right|^k\right] 
     \le &\,   \left( \int_0^t ds \int_{\mathbb R^d}d\z     \left|G^{\alpha}_{t-s}(\x, \z)-G^{\alpha}_{t-s}(\y, \z)\right|  \right)^{k-1}\\
   &\,\,\, \times \left\{  \int_0^t ds \int_{\mathbb R^d}d\z \left|G^{\alpha}_{t-s}(\x, \z)-G^{\alpha}_{t-s}(\y, \z)\right|\cdot  \mathbb E\left[ \left|b(u_s(\z))\right|^k\right]\right\}\\
  \le &\,   c_{2,15} \|\x-\y\|^{k-1}\cdot  \left(1+\sup_{(s,\z)\in[0,T]\times \mathbb R^d}\mathbb E\left[ \left|u_s(\z)\right|^k\right]\right),
  \end{split}
  \end{equation*}
  where $c_{2,15}>0$.
  This, together with \eqref{eq u moment}, completes the proof.
        \end{proof}

     \section{Localization}

 For any $\e>0$ and any unit vector  $\boldsymbol e$ in $\mathbb R^d$,   define
       \begin{align}\label{eq nabla}
        (\nabla_{\e\boldsymbol  e} f)(\x):= f(\x)-f(\x-\e \boldsymbol  e)   \ \ \ \ \  (\x \in\mathbb R^d),
          \end{align}
          as a substitute for the approximate  spatial gradient of the  function $f:\mathbb R^d\rightarrow\mathbb R$.

     Let us recall the solution $Z$ to Eq.  \eqref{eq mild SHE}.
  Consider the approximate gradient operator
     \begin{align}\label{eq nabla Z}
    (\nabla_{\e\boldsymbol  e} Z_t)(\x):=Z_t(\x)-Z_{t}(\x-\e\boldsymbol e)=\int_{(0,t)\times\mathbb R^d}\left(\nabla_{\e\boldsymbol  e}  G^{\alpha}_{t-s}\right)(\x, \y)  M(ds,d\y).
     \end{align}

         Throughout, let us choose and fix  parameters
     \begin{align}\label{eq beta delta}
     \beta>1\ \ \ \ \text{and } \ \ \ \  \delta:=1+\beta^{ 1+\frac{1}{H}}.
     \end{align}
     Then, we define a family of space-time boxes as follows: For every $t\ge0, \x\in \mathbb R^d$ and $\varepsilon>0$,
     \begin{align}\label{eq Box}
     \mathbf B_{\beta}(\x, t; \e)=[t-\beta \e^{\alpha}, t]\times  \mathbf B(\x; \e \delta),
     \end{align}
     where $\mathbf B(\x; \e \delta)=\{\y\in \mathbb R^d:\, \|\x-\y\|\le \e\delta\}$.

     The following is a generalization of Proposition 4.1 in \cite{FKM2015}.
     \begin{proposition}\label{prop local} Choose and fix  $T>0$. There exists  a   positive constant $c_{3,1}$ such that
     \begin{equation}\label{eq local appr}
     \left\|(\nabla_{\e \boldsymbol e} Z_t)(\x)-\int_{ \mathbf B_{\beta}(\x, t; \e)} \left(    \nabla_{\e\boldsymbol  e}   G^{\alpha}_{t-s}\right)(\x, \y) M(ds,d\y)\right\|_{L^2(\Omega)}\le c_{3,1} \e^{H} \beta^{-\frac{d+2-\alpha-\gamma}{2\alpha}},
     \end{equation}
     simultaneously for all $\x\in \mathbb R^d$, $t\in [0,T]$, $\e\in (0,1)$, $\beta>1$ and the unit vector $\boldsymbol e$ in $\mathbb R^d$.
          \end{proposition}
    \begin{proof}  By the stationary of $\{Z_t(\x)\}_{\x\in \mathbb R^d}$ and the independence of $\sigma\{M(s,\x); s\in [0,t], \x\in \mathbb R^d\}$ and $\sigma\{M(s,\x)-M(t,\x); s\ge t, \x\in \mathbb R^d\}$, we have
    \begin{equation*}
    \begin{split}
Q:=  &\,  \mathbb E\left[\left|( \nabla_{\e\boldsymbol e} Z_t)(\x) -\int_{ \mathbf B_{\beta}(\x, t; \e)} \left(    \nabla_{\e\boldsymbol e} G^{\alpha}_{t-s} \right)(\x,\y)   M(ds,d\y)\right|^2 \right]\\
 = &\,   \mathbb E\left[ \left|\int_{[(0,t)\times \mathbb R^d]\backslash   \mathbf B_{\beta}(0, t; \e)}    \left(    \nabla_{\e\boldsymbol e} G^{\alpha}_{t-s} \right)(\y)   M(ds,d\y)\right|^2 \right]\\
     =&\, \mathbb E\left[\left|\int_0^{t-\beta\e^{\alpha}}\int_{\mathbb R^d } \left( \nabla_{\e\boldsymbol e} G^{\alpha}_{t-s} \right)(\y)   M(ds,d\y)\right|^2 \right]\\
       &\,+ \mathbb E\left[\left|\int_{t-\beta\e^{\alpha}}^t\int_{\|\y\|> \e\delta}  \left( \nabla_{\e\boldsymbol e} G^{\alpha}_{t-s} \right)(\y)   M(ds,d\y)\right|^2 \right]\\
       =:&\, Q_1+Q_2.
    \end{split}
         \end{equation*}
 Because $\mathcal F\left(\nabla_{\e\boldsymbol  e} G^{\alpha}_{t-s} \right)(\cdot)(\xi)=e^{-(t-s)\|\xi\|^{\alpha}}\left(1-e^{-2\pi\e i \xi\cdot \boldsymbol  e}\right)$,    \eqref{eq kernel 1} implies that
   \begin{equation}\label{eq Q1}
    \begin{split}
    Q_1=&\,  \int_{\beta\e^{\alpha}}^t ds \int_{\mathbb R^d } d\xi \, e^{-2s\|\xi\|^{\alpha}}\left[1- \cos\left(2\pi\e   \xi\cdot \boldsymbol  e\right)  \right]\|\xi\|^{-\gamma}\\
    \le &\,   2\pi^2 \e^2\int_{\beta\e^{\alpha}}^{\infty}ds \int_{\mathbb R^d } d\xi \, e^{-2s\|\xi\|^{\alpha}} \|\xi\|^{2-\gamma}\\
     = &\,    \pi^2 \e^2  \int_{\mathbb R^d }     e^{-2\beta\e^{\alpha}\|\xi\|^{\alpha}} \|\xi\|^{2-\alpha-\gamma} d\xi \\
        =&\   \pi^2 \int_{\mathbb R^d } e^{- \|\xi\|^{\alpha}} \|\xi\|^{2-\alpha-\gamma}d\xi \cdot (2\beta)^{-\frac{d+2-\alpha-\gamma}{\alpha}}    \e^{\alpha-d+\gamma},
        \end{split}
         \end{equation}
     where   $1-\cos(a)\le a^2/2$ ($a\in \mathbb R$)  has been used in the second step.

    Let us estimate $Q_2$ with a litter more effort.     First,  by the  H\"older  inequality, we have that for  any  $p>\max\left\{\frac{2\alpha  }{\alpha-d+\gamma},\frac{d}{\gamma}\right\}$,
      \begin{equation}\label{eq Q2}
    \begin{split}
    Q_2=&\,  \int_0^{\beta\e^{\alpha}} ds \int_{\|\y\|>\e\delta}d\y \int_{\|\z\|>\e\delta}  d\z  \left(   \nabla_{\e \boldsymbol  e} G^{\alpha}_{s} \right)(\y) \|\y-\z\|^{-d+\gamma}\left(   \nabla_{\e \boldsymbol  e} G^{\alpha}_{s} \right)(\z)      \\
   \le &\,   \left\{\int_0^{\beta \e^{\alpha}}ds  \int_{\|\y\|>\e\delta}d\y \int_{\|\z\|>\e\delta} d\z \left| \left(   \nabla_{\e \boldsymbol  e} G^{\alpha}_{s} \right)(\y)\right|   \left|\left(   \nabla_{\e \boldsymbol  e} G^{\alpha}_{s} \right)(\z)\right|    \right\}^{1/p}\\
 & \times \left\{ \int_0^{\beta \e^{\alpha}}ds \int_{\mathbb R^d}d\y \int_{\mathbb R^d} d\z  \left|\left(   \nabla_{\e \boldsymbol  e} G^{\alpha}_{s} \right)(\y)\right|\|\y-\z\|^{-\frac{p}{p-1}(d-\gamma)}\left|\left(   \nabla_{\e \boldsymbol  e} G^{\alpha}_{s} \right)(\z)\right|\right\}^{\frac{p}{p-1}}\\
 =:&\, Q_{2,1}\cdot Q_{2,2}.
     \end{split}
         \end{equation}

    Next,  we use the argument in the proof of (4.15) in \cite{FKM2015} to estimate $Q_{2,1}$.
           Notice that
        \begin{equation}\label{eq Q21}
    \begin{split}
    \int_{\|\y\|>\e\delta}\left| \left(   \nabla_{\e \boldsymbol  e} G^{\alpha}_{s} \right)(\y)\right|d\y
   \le\, \int_{\|\y\|>\e\delta}\left|  G^{\alpha}_{s}  (\y)\right|d\y+\int_{\|\y\|>\e\delta}\left|    G^{\alpha}_{s}  (\y-\e  \boldsymbol e)\right|d\y.
              \end{split}
         \end{equation}
         Let $L^{(\alpha)}=\left\{L^{(\alpha)}_t\right\}_{t\ge0}$ be the  rotationally   invariant $\alpha$-stable distribution.
    If $\eta\in [0,1]$, then
          \begin{equation}\label{eq Q211}
      \int_{\|\y\|>\e\delta}\left|  G^{\alpha}_{s}  (\y-\eta \e \boldsymbol e)\right|d\y\le \mathbb P\left\{|L^{(\alpha)}_s|\ge \e(\delta-\eta)\right\}.
        \end{equation}
    Since $\delta=1+\beta^{1+\frac{1}{H}}$ [see \eqref{eq beta delta}], by the triangle inequality and the scaling of $L^{(\alpha)}$, we have
       \begin{equation}\label{eq Q212}
         \begin{split}
    \sup_{0\le \eta \le 1}  \int_{\|\y\|>\e\delta}\left|  G^{\alpha}_{s}  (\y-\eta \e \boldsymbol e)\right|d\y\le &\, \mathbb P\left\{|L^{(\alpha)}_s|\ge \e\beta^{1+\frac{1}{H}}\right\}\\
    =&\, \mathbb P\left\{|L^{(\alpha)}_1|\ge \e \beta^{1+\frac{1}{H}}  s^{-1/\alpha}\right\}\\
    \le &\, \e^{-\alpha}\beta^{-\alpha(1+\frac{1}{H})} s,
        \end{split}
        \end{equation}
  where we have used a well-known bound on the tail of the rotationally invariant $\alpha$-stable distribution
  $$
  \mathbb P\left(\left\|L^{(\alpha)}_1\right\|>\lambda \right)\le \text{const} \cdot \lambda^{-\alpha} \ \ \ \text{for all }\ \lambda>0.
  $$
   See, e.g., \cite{Appl}.

       By \eqref{eq Q21}-\eqref{eq Q212}, we have
        \begin{equation}\label{eq Q213}
         \begin{split}
         Q_{21}\le  \,  \left(3^{-1}  \e^{\alpha }\beta^{-\left(2\alpha\left(1+\frac{1}{H}\right)-3\right)}\right)^{1/p}.
     \end{split}
        \end{equation}

          By    using  elementary inequality of
 \begin{align}\label{eq elem}
 2|\langle f, g\rangle_{\mathcal H}|\le  \|f\|_{\mathcal H}^2+ \|g\|_{\mathcal H}^2,
 \end{align}
  we have
       \begin{equation}\label{eq kernel Q22}
 \begin{split}
    Q_{22} = &\left\{\int_0^{\beta \e^{\alpha}}ds \int_{\mathbb R^d}d\y\int_{\mathbb R^d} d\z  \left|\left(   \nabla_{\e \boldsymbol e} G^{\alpha}_{s} \right)(\y)\right|\|\y-\z\|^{-\frac{p}{p-1}(d-\gamma)}\left|\left(   \nabla_{\e \boldsymbol e} G^{\alpha}_{s} \right)(\z)\right|\right\}^{\frac{p-1}{p}}\\
    \le &\, \left\{4 \int_0^{\beta \e^{\alpha}}ds \int_{\mathbb R^d}d\y\int_{\mathbb R^d} d\z \, G^{\alpha}_{s}  (\y)\|\y-\z\|^{- d+\frac{1}{p-1}(p\gamma-d)}  G^{\alpha}_{s} (\z) \right\}^{\frac{p-1}{p}}\\
    =&\, \left\{4 c_{3,2} \int_0^{\beta \e^{\alpha}}ds\int_{\mathbb R^d}d\xi\, \|\xi\|^{-\frac{1}{p-1}(p\gamma-d) }\, \left|\mathcal F G^{\alpha}_s(\cdot)(\xi)\right|^2\right\}^{\frac{p-1}{p}}\\
    =&\, \left\{4 c_{3,2} \int_0^{\beta \e^{\alpha}}ds\int_{\mathbb R^d}d\xi\, \|\xi\|^{-\frac{1}{p-1}(p\gamma-d) }\,e^{-2s\|\xi\|^{\alpha}} \right\}^{\frac{p-1}{p}} \\
     =&\, \left\{4 c_{3,2} c_{3,3}  \int_0^{\beta \e^{\alpha}}s^{-\frac{p}{(p-1)\alpha}(d-\gamma) } ds\right\}^{\frac{p-1}{p}}\\
      =&\,  \left[ 1-\frac{p(d-\gamma)}{(p-1)\alpha}  \right]^{\frac{p-1}{p}} \left(4 c_{3,2} c_{3,3}\right)^{{\frac{p-1}{p}}} (\beta \e^{\alpha})^{\frac{p-1}{p}-\frac{d-\gamma}{\alpha}},
            \end{split}
 \end{equation}
     where  the constant $c_{3,2}\in (0,\infty)$ is   of the form $c_{1,1}$ in \eqref{eq c11} replacing $\gamma$ by $(p\gamma-d)/(p-1)$,
     and $c_{3,3}=\int_{\mathbb R^d} \|\xi\|^{-\frac{1}{p-1}(p\gamma-d)}\, e^{-2 \|\xi\|^{\alpha}}d\xi<\infty$.

    We put   \eqref{eq Q213} and \eqref{eq kernel Q22} together in order to see that
     \begin{equation}\label{eq Q2 31}
     Q_{2}\le \text{const}\cdot \e^{\alpha-d+\gamma}\beta^{-\left(2\alpha(1+\frac{2}{\alpha-d+\gamma})-3\right)/p+\left(p-1\right)/p-(d-\gamma)/\alpha}.
      \end{equation}

     Finally, we can combine our bounds for $Q_{1}$ and $Q_{2}$ in order deduce the inequality,
          \begin{equation}
     Q \le \text{const}\cdot \e^{\alpha-d+\gamma}\left(\beta^{-\frac{d+2-\alpha-\gamma}{\alpha}} + \beta^{-\left(2\alpha(1+\frac{2}{\alpha-d+\gamma})-3\right)/p+\left(p-1\right)/p-(d-\gamma)/\alpha}\right).
      \end{equation}
     Elementary analysis of the exponent of $\beta$ shows that
     $$
     -\frac{d+2-\alpha-\gamma}{\alpha}\ge  -\frac{1}{p}\left(2\alpha\left(1+\frac{2}{\alpha-d+\gamma}\right)-3\right)+\frac{p-1}{p}-\frac{d-\gamma}{\alpha},
     $$
     because  $p\ge\frac{2\alpha}{\alpha-d+\gamma }$. Since $\beta>1$,  we get the desired result \eqref{eq local appr}.

     The proof is complete.
             \end{proof}

     Since Corollary \ref{coro Z spatial eq} says that $\mathbb E\left(\left|(\nabla_{\e \boldsymbol  e} Z_t)(\x) \right|^2 \right)\approx \text{const}\cdot \e^{\alpha-d+\gamma} $ when $\e\ll1$, Proposition \ref{prop local} shows that
     \begin{align}
     \nabla_{\e \boldsymbol  e} Z_t(\x)\approx \int_{ \mathbf B_{\beta}(\x, t; \e)} \left( \nabla_{\e \boldsymbol e}   G^{\alpha}_{t-s}\right)(\x, \y)   M(ds,d\y)\ \ \ \text{when } \ \beta\gg1 \text { and } \, \e\ll 1.
     \end{align}
   That is,  the most of contribution to the stochastic integral in \eqref{eq nabla Z} comes from the region $\mathbf B_{\beta}(\x, t; \e)$.  This is  the so-called   {\it strong localization} property  in \cite{FKM2015}.

  The next result describes  the  strong localization property   the solution   to the   nonlinear heat equation \eqref{eq SFBE}.
 \begin{lemma}\label{coro local appr2}
 Choose and fix  $T>0$ and $k\in [2,k_0]$. Then, there exists a   constant $c_{3,4}>0$ such that
     \begin{equation}\label{eq local appr2}
     \begin{split}
    & \left\|(\nabla_{\e \boldsymbol  e} X_t)(\x)-\int_{ \mathbf B_{\beta}(\x, t; \e)} \left(    \nabla_{\e \boldsymbol  e}   G^{\alpha}_{t-s}\right)(\x, \y)\sigma(u_s(\y))   M(ds,d\y)\right\|_{L^k(\Omega)} \\
     \le&\,  c_{3,4} \e^{H}  \beta^{-\frac{2+d-\alpha-\gamma}{2\alpha}},
     \end{split}
     \end{equation}
     simultaneously for all $\x\in \mathbb R^d$, $t\in[0,T]$, $\e\in (0,1)$, $\beta>1$ and the unit vector $\boldsymbol  e$ in $\mathbb R^d$.
  \end{lemma}
 
        \begin{proof}  { As in the proof of Proposition \ref{prop local}, we have}  { 
    \begin{equation*}
    \begin{split}
Q:=  &\, \left\| ( \nabla_{\e\boldsymbol e} X_t)(\x) -\int_{ \mathbf B_{\beta}(\x, t; \e)} \left(    \nabla_{\e\boldsymbol e} G^{\alpha}_{t-s} \right)(\x,\y)   \sigma(u_s(\y))M(ds,d\y)  \right\|_{L^k(\Omega)}^2\\
 = &\,   \left\|\int_{[(0,t)\times \mathbb R^d]\backslash   \mathbf B_{\beta}(0, t; \e)}    \left(    \nabla_{\e\boldsymbol e} G^{\alpha}_{t-s} \right)(\y) \sigma(u_s(\y))  M(ds,d\y) \right\|_{L^k(\Omega)}^2\\
     \le&\,\left\|\int_0^{t-\beta\e^{\alpha}}\int_{\mathbb R^d } \left( \nabla_{\e\boldsymbol e} G^{\alpha}_{t-s} \right)(\y) \sigma(u_s(\y))  M(ds,d\y) \right\|_{L^k(\Omega)}^2\\
       &\,+ \left\|\int_{t-\beta\e^{\alpha}}^t\int_{\|\y\|> \e\delta}  \left( \nabla_{\e\boldsymbol e} G^{\alpha}_{t-s} \right)(\y) \sigma(u_s(\y))  M(ds,d\y) \right\|_{L^k(\Omega)}^2\\
       =:&\, Q_1+Q_2.
    \end{split}
         \end{equation*}
         }
  { 
By using \eqref{eq u moment} and using the same argument in the proof of \eqref{eq Q2 31}, we can obtain that
    \begin{align}\label{eq L Q2}
    Q_2\le c_{3,5} \e^{  \alpha-d+\gamma  } \beta^{-\frac{ 2+d-\alpha-\gamma}{\alpha}},
     \end{align}
     where $c_{3,5}>0$.}
  { 
 By using the Burkholder-Davis-Gundy inequality \eqref{eq BDG},  \eqref{eq Lip}, \eqref{eq u moment}, \eqref{eq Green 5} and \eqref{eq elem},
  we have }
  { 
             \begin{equation} \label{eq L Q1}
             \begin{split}
     Q_1\le &\,    4k \int_0^{t-\beta \e^{\alpha}}ds \int_{\mathbb R^d}d\y\int_{\mathbb R^d} d\z \left| \left(\nabla_{\e\boldsymbol e}   G^{\alpha}_{t-s}\right)(  \y)  \right|   \left\| \sigma(u_s(\y)) \right\|_{L^k(\Omega)}\\
     &\,\,\,\, \,\,\,\, \cdot  \|\y-\z\|^{-(d-\gamma)} \cdot \left| \left(\nabla_{\e\boldsymbol e}  G^{\alpha}_{t-s}\right)(  \z)\right|   \left\| \sigma(u_s(\z))   \right\|_{L^k(\Omega)}\\
     \le&\,  c_{3,6}  \e^2\int_0^{t-\beta \e^{\alpha}}  (t-s)^{-\frac{2}{\alpha}}ds\,   \int_{\mathbb R^d}d\y\int_{\mathbb R^d} d\z     \left(  G^{\alpha}_{t-s} (\y)+  G^{\alpha}_{t-s} ( \y-\e\boldsymbol e)   \right) \\
     &\,\,\,\, \,\,\,\, \cdot  \|\y-\z\|^{-(d-\gamma)} \cdot     \left(  G^{\alpha}_{t-s} (\z)+  G^{\alpha}_{t-s} ( \z-\e\boldsymbol e)   \right) \\
          \le&\, 4 c_{3,6} \e^2\int_0^{t-\beta \e^{\alpha}}  (t-s)^{-\frac{2}{\alpha}}ds\,   \int_{\mathbb R^d}d\y\int_{\mathbb R^d} d\z    \,   G^{\alpha}_{t-s} (\y)\cdot  \|\y-\z\|^{-(d-\gamma)} \cdot     G^{\alpha}_{t-s} (\z) \\
          =&\, 4c_{1,1}^{-1}c_{3,6}   \e^2\int_0^{t-\beta \e^{\alpha}}  (t-s)^{-\frac{2}{\alpha}}ds\,
 \int_{\mathbb R^d}d\xi\, \|\xi\|^{-\gamma}\, e^{-2(t-s)  \|\xi\|^{\alpha}}\\
              \le &\,  4c_{1,1}^{-1}c_{3,6}   \e^2  \int_{\beta \e^{\alpha}}^{\infty}  s^{-\frac{2+d-\gamma}{\alpha}}ds\cdot \int_{\mathbb R^d} \|\xi\|^{-\gamma}\, e^{-2 \|\xi\|^{\alpha}}d\xi\\
     = &\, \Big(4c_{1,1}^{-1}c_{3,6}  \int_{\mathbb R^d} \|\xi\|^{-\gamma}\, e^{-2 \|\xi\|^{\alpha}}d\xi \Big) \e^{  \alpha-d+\gamma  } \beta^{-\frac{2+d-\alpha-\gamma}{\alpha}},
        \end{split}
        \end{equation}
 where $c_{3,6}>0$. }
  {Putting \eqref{eq L Q2} and  \eqref{eq L Q1} together  and using  the Minkowski  inequality,   we obtain \eqref{eq local appr2}.
                 The proof is complete.  }     \end{proof}

     \begin{lemma}\label{lem local appr3} Choose and fix $T>0$, $k\in [2,k_0]$ and   $\beta:=\e^{-l}$ for some $l>0$. Then, there exists a  constant $c_{3,7}>0$ such that
       \begin{equation}\label{eq local appr3}
     \begin{split}
    &  \left\| \int_{ \mathbf B_{\beta}(\x, t; \e)} \left(\nabla_{\e\boldsymbol e}   G^{\alpha}_{t-s}\right)(\x, \y)\left[\sigma(u_s(\y))- {\sigma(u_t(\tilde \x))} \right]  M(ds,d\y)\right\|_{L^k(\Omega)} \\
     \le&\,  c_{3,7} \e^{2H} \beta^{1+\alpha-d+\gamma},
     \end{split}
     \end{equation}
         simultaneously for all $\x\in \mathbb R^d$,  $t\in [0,T]$, $\e\in (0,1)$, $l>0$, $\tilde \x\in \mathbf B(\x; \e \delta)$      and the unit vector $\boldsymbol e$ in $\mathbb R^d$.
       \end{lemma}
       \begin{proof}           
  { Let $t\in [0,T]$ and $\tilde \x\in  \mathbf B(\x; \e \delta)$.
       We first consider a related quantity $Q_1$, where
                     \begin{equation}\label{eq local appr4}
     \begin{split}
     Q_1:= \left\| \int_{ \mathbf B_{\beta}(\x, t; \e)} \left(    \nabla_{\e\boldsymbol e}   G^{\alpha}_{t-s}\right)(\x, \y)\left[\sigma(u_s(\y))-\sigma(u_{t-\beta\e^{\alpha}}(\tilde \x)) \right]  M(ds,d\y)\right\|_{L^k(\Omega)}^2.
     \end{split}
     \end{equation}
              }
  { 
       Since $u_{t-\beta\e^{\alpha}}(\tilde \x)$ is measurable with respect to the time-white noise of $[0,t-\beta\e^{\alpha}]\times \mathbb R^d$, it is independent of the time-white noise of $\mathbf B_{\beta}(\x, t; \e)$. Therefore, by the Burkholder-Davis-Gundy inequality \eqref{eq BDG}, we have
             \begin{align*}
     Q_1\le &\,    2k \int_{t-\beta \e^{\alpha}}^tds \int_{\mathbf B(\x; \e \delta)}d\y\int_{\mathbf B(\x; \e \delta)} d\z \left(\nabla_{\e\boldsymbol e}   G^{\alpha}_{t-s}\right)(\x, \y)     \left\| \sigma(u_s(\y))-\sigma(u_{t-\beta\e^{\alpha}} (\tilde \x)) \right\|_{L^k(\Omega)}\\
     &\,\,\,\, \,\,\,\, \cdot  \|\y-\z\|^{-(d-\gamma)} \cdot  \left(\nabla_{\e\boldsymbol e}  G^{\alpha}_{t-s}\right)(\x, \z)   \left\| \sigma(u_s(\z))-\sigma(u_{t-\beta\e^{\alpha}} (\tilde \x)) \right\|_{L^k(\Omega)}.
        \end{align*}         }
  { 
        By H\"older's inequality, we have that for any $p>\max\{\frac{d}{\gamma}, \frac{\alpha^2}{\alpha-d+\gamma}\}$,
               \begin{align*}
     Q_1
     \le &\,  2k  \int_{t-\beta \e^{\alpha}}^tds\left[ \int_{\mathbb R^d}d\y\int_{\mathbb R^d} d\z \left|\left(\nabla_{\e\boldsymbol e}   G^{\alpha}_{t-s}\right)(\x, \y) \right|\cdot \left|\left(\nabla_{\e\boldsymbol e}  G^{\alpha}_{t-s}\right)(\x, \z) \right|\right]^{\frac{1}{p}}\\
      &\,\,\,\, \,\,\,\,
     \cdot   \sup_{(s, \y)\in \mathbf B_{\beta}(\x, t; \e) }  \left\| \sigma(u_s(\y))-\sigma(u_{t-\beta\e^{\alpha}} (\tilde \x)) \right\|_{L^k(\Omega)}^2 \\
     &\,\,\,\, \,\,\,\, \cdot\left[\int_{\mathbb R^d}d\y\int_{\mathbb R^d} d\z \left|\left(\nabla_{\e\boldsymbol e}   G^{\alpha}_{t-s}\right)(\x, \y) \right|\cdot \|\y-\z\|^{-\frac{p}{p-1}(d-\gamma)}\cdot\left|  \left(\nabla_{\e\boldsymbol e}  G^{\alpha}_{t-s}\right)(\x, \z)\right|\right]^{\frac{p}{p-1}}.
        \end{align*}         }
  { 
      By     \eqref{eq Green 5} and \eqref{eq elem}, we have
       \begin{align*}
    &  \int_{\mathbb R^d}d\y\int_{\mathbb R^d} d\z
    \left|\left(\nabla_{\e\boldsymbol e}   G^{\alpha}_{t-s}\right)(\x, \y)  \right| \cdot \left| \left(\nabla_{\e\boldsymbol e}  G^{\alpha}_{t-s}\right)(\x, \z) \right|\\
      \le&\,  c_{2,2}\frac{\e}{(t-s)^{\alpha}} \int_{\mathbb R^d}d\y\int_{\mathbb R^d} d\z   \left( G^{\alpha}_{t-s} (\x-\y) +G^{\alpha}_{t-\e-s} (\x-\y)\right)  \left( G^{\alpha}_{t-s} (\x-\z) +G^{\alpha}_{t-\e-s} (\x-\z)\right) \\
            \le&\,  4c_{2,2}\frac{\e}{(t-s)^{\alpha}}.
       \end{align*}         }
  { 
         By  Proposition \ref{prop Holder}, we have that for
     \begin{align*}
         \sup_{(s, \y)\in \mathbf B_{\beta}(\x, t; \e) }  \left\| \sigma(u_s(\y))-\sigma(u_{t-\beta\e^{\alpha}} (\tilde \x)) \right\|_{L^k(\Omega)}^2 \le c_{2,8} \e^{2H} \left[ \beta^{H/\alpha} +(2\delta)^{H}\right]^2.           \end{align*}
 By  \eqref{eq elem}, we obtain that
    \begin{align*}
   & \int_{\mathbb R^d}d\y\int_{\mathbb R^d} d\z \left|\left(\nabla_{\e\boldsymbol e}   G^{\alpha}_{t-s}\right)(\x, \y) \right|\cdot \|\y-\z\|^{-\frac{p}{p-1}(d-\gamma)}\cdot \left|  \left(\nabla_{\e\boldsymbol e}  G^{\alpha}_{t-s}\right)(\x, \z)\right|\\
    \le &\, 4 \int_{\mathbb R^d}d\y\int_{\mathbb R^d} d\z\,    G^{\alpha}_{t-s}(\x-\y)\|\y-\z\|^{-\frac{p}{p-1}(d-\gamma)} G^{\alpha}_{t-s}(\x-\z)\\
= &\,  c_{3,8}\int_{\mathbb R^d}d\xi\, \|\xi\|^{-\frac{1}{p-1}(p\gamma-d) }\,e^{-2(t-s)\|\xi\|^{\alpha}}  \\
= &\, \left(  c_{3,8}  \int_{\mathbb R^d} \, \|\xi\|^{-\frac{p\gamma-d}{p-1}  }\,e^{-2 \|\xi\|^{\alpha}}d\xi\right)\cdot (t-s)^{-\frac{p(d-\gamma)}{(p-1)\alpha}},
       \end{align*}         }
  { 
       where $c_{3,8}\in(0,\infty)$,   the integral $\int_{\mathbb R^d} \, \|\xi\|^{-\frac{1}{p-1}(p\gamma-d) }\,e^{-2 \|\xi\|^{\alpha}}d\xi$ is finite   because    $\frac{p\gamma-d}{p-1}<d$.
  Thus,         }
  { 
              \begin{equation}\label{eq local appr7}
              \begin{split}
   Q_1
     \le&\, c_{3,9}   \e^{2H+\frac{1}{p}} \left[ \beta^{H/\alpha}+(2\delta)^{H}\right]^2   \int_{t-\beta \e^{\alpha}}^t \,(t-s)^{-\frac{\alpha}{p}-\frac{d-\gamma}{\alpha}}  ds \\
      = &\, c_{3,9} \left(1- \frac{\alpha}{p}-\frac{d-\gamma}{\alpha}\right)^{-1}\cdot  \e^{2(\alpha-d+\gamma)}\left[ \beta^{H/\alpha}+(2\delta)^{H}\right]^2 \cdot   \beta^{1-\frac{d-\gamma}{\alpha}-\frac{\alpha}{p}} \e^{-\frac{\alpha^2}{p}+\frac{1}{p}}.
           \end{split}
           \end{equation}
                     }
                     
  {Next, we estimate the cost of estimating $u_t(\tilde \x)$ by $u_{t-\beta\e^{\alpha}}(\tilde \x)$. By the H\"older inequality,  \eqref{eq u xt}  and  Lemma \ref{lem Z spatial}(a), we have         }
  { 
  \begin{equation}\label{eq Q2 1}
     \begin{split}
     Q_2 := &\,    \left\|\left[\sigma\left(u_t(\tilde \x)\right)-\sigma\left(u_{t-\beta\e^{\alpha}}(\tilde \x)\right) \right]  \int_{ \mathbf B_{\beta}(\x, t; \e)} \left(    \nabla_{\e\boldsymbol e}   G^{\alpha}_{t-s}\right)(\x, \y)   M(ds,d\y)\right\|_k^2\\
     \le &\,  L_{\sigma}^2 \left\| u_t(\tilde \x)-u_{t-\beta\e^{\alpha}}(\tilde \x) \right\|^2_{2k}\cdot  \left\| \int_0^t\int_{\mathbb R^d} \left(    \nabla_{\e\boldsymbol e}   G^{\alpha}_{t-s}\right)(\x, \y)   M(ds,d\y)\right\|_{2k}^2\\
     \le&\,  c_{2,7}^2 L_{\sigma}^2 \left( \beta\e^{\alpha}  \right)^{2H/\alpha} \cdot   \e^{\alpha-d+\gamma}= c_{2,7}^2 L_{\sigma}^2 \e^{2(\alpha-d+\gamma)}\beta^{\frac{\alpha-d+\gamma}{\alpha}}.
                \end{split}
      \end{equation}   
            }
  { 
       Since $\delta=1+\beta^{1+\frac{1}{H} }\le 2\beta^{1+\frac{1}{H}}$,  we can conclude  from \eqref{eq local appr7} and \eqref{eq Q2 1} that
       \begin{align*}
       \sqrt{Q_1}+ \sqrt{Q_2} 
       \le c_{3,10} \e^{\alpha-d+\gamma}\beta^{1+\alpha-d+\gamma}\cdot  \left[\beta^{-\frac{(\alpha-1)(\alpha-d+\gamma)}{2\alpha}-\frac{\alpha}{2p}} \e^{-\frac{\alpha^2}{2p}+\frac{1}{2p}}+\beta^{-1-(\alpha-d+\gamma)+\frac{\alpha-d+\gamma}{2\alpha})}\right],
       \end{align*}
       where $c_{3,10}\in (0,\infty)$. Since $\beta=\e^{-l}$ for some $l>0$,   taking $p$ large enough such that $\beta^{-\frac{(\alpha-1)(\alpha-d+\gamma)}{2\alpha}-\frac{\alpha}{2p}} \e^{-\frac{\alpha^2}{2p}+\frac{1}{2p}}<1$, we obtain that
       \begin{align*}
       \sqrt{Q_1}+ \sqrt{Q_2}\le 2c_{3,10} \e^{\alpha-d+\gamma}\beta^{1+\alpha-d+\gamma}.
       \end{align*}        
        }
  { This and Minkowski's inequality together imply this  lemma.     }   \end{proof}

    \section{Proof of Theorem \ref{thm spatial gradient}}

         \begin{proof}[Proof of Theorem \ref{thm spatial gradient}]

         According to the proof of Theorem 1.1 in \cite{FKM2015},  it is sufficient to get the  following estimate:   For any fixed $T_1>T_0>0$,
     \begin{equation}\label{eq q1}
     \begin{split}
    \sup_{t\in[T_0,T_1]}\sup_{\x\in \mathbb R^d} \sup_{\|\boldsymbol  {e}\|=1  } \left\|\left(\nabla_{\e \boldsymbol  e}u_t\right)(\x)-c_{\alpha, d,\gamma}\sigma(u_t(\tilde \x))  \left(\nabla_{\e  \boldsymbol  e}B^{H}\right)(\x)\right\|_k
    \le \,  A \e^{H+\eta},
     \end{split}
     \end{equation}
where  $ \tilde \x\in \mathbf B(\x; \e \delta), k\in [2,k_0], A>0$ and $\eta>0$.

Recall the decomposition of $u$ in \eqref{eq decomposition}.  By using the estimates in  \eqref{eq u0 3} and \eqref{eq regul Y}, we know that to prove \eqref{eq q1}, it suffices to prove that
     \begin{equation}\label{eq q2}
     \begin{split}
    \sup_{t\in[T_0,T_1]} \sup_{\x\in \mathbb R^d} \sup_{\| \boldsymbol {e}\|=1  } \left\|\left(\nabla_{\e \boldsymbol  e}X_t\right)(\x)-c_{\alpha, d,\gamma}\sigma(u_t(\tilde \x))  \left(\nabla_{\e \boldsymbol   e}B^{H}\right)(\x)\right\|_k 
    \le    A \e^{H+\eta},
     \end{split}
     \end{equation}
  where  $ \tilde \x\in \mathbf B(\x; \e \delta), k\in [2,k_0], A>0$ and $\eta>0$.

  Let us split the preceding expectation in two parts as follows:
 $$
 \left\|\left(\nabla_{\e \boldsymbol  e}X_t\right)(\x)-c_{\alpha, d,\gamma}\sigma(u_t(\tilde \x))  \left(\nabla_{\e \boldsymbol  e}B^{H}\right)(\x)\right\|_k
\le I_1+I_2,
 $$
     where
      \begin{align}\label{eq qI1}
I_1:= &\,  \left\|\left(\nabla_{\e\boldsymbol e}X_t\right)(\x)- \sigma(u_t(\tilde \x))  \left(\nabla_{\e \boldsymbol  e} Z_t\right)(\x)\right\|_k; \\
I_2:=&\,  \left\|\sigma(u_t(\tilde \x))  \left(\nabla_{\e\boldsymbol e} Z_t\right)(\x) -c_{\alpha, d, \gamma}\sigma(u_t(\tilde \x))   \left(\nabla_{\e \boldsymbol  e} B^{H}\right)(\x)\right\|_k.  \label{eq qI2}
     \end{align}

   By    Lemma \ref{lem Z spatial}(a) and Proposition  \ref{prop Holder}, there exists a constant $c_{4,1}>0$ such that
 \begin{equation}\label{eq qI11}
     \begin{split}
            \left\|   \left(\nabla_{\e \boldsymbol  e} Z_t\right)(\x) \big[\sigma(u_t(\tilde \x))-\sigma(u_t(\x)) \big]\right\|_k
            \le &\,  L_{\sigma} \left\|   \left(\nabla_{\e \boldsymbol  e} Z_t\right)(\x)  \right\|_{2k} \cdot \left\|u_t(\tilde \x)-u_t(\x) \right\|_{2k}\\
           \le &\, c_{4,1} \e^{\alpha-d+\gamma} \beta^{1+\frac{\alpha-d+\gamma}{2}}.
                    \end{split}
              \end{equation}
          By Corollary \ref{coro local appr2} and Lemma \ref{lem local appr3}, we have
                          \begin{equation}\label{eq qI12}
     \begin{split}
           &\left\|\left(\nabla_{\e \boldsymbol  e}X_t\right)(\x)- \sigma(u_t(\x))  \left(\nabla_{\e \boldsymbol  e} Z_t\right)(\x) \right\|_k \\
            \le  &\, c_{4,2}\e^{\frac{\alpha-d+\gamma}{2}} \beta^{-\frac{d+2-\alpha-\gamma}{2\alpha}}+c_{4,3}\e^{\alpha-d+\gamma}\beta^{1+\frac{\alpha-d+\gamma}{2}}.
                    \end{split}
              \end{equation}
      where $c_{4,2}$ and $c_{4,3}$ denote  the finite constants   that do  not depend on the values of $\x\in \mathbb R^d$, $t\in [T_0,T_1]$, $\e\in (0,1)$ and the unit vector $\boldsymbol {e}$ in $\mathbb R^d$.

      Define $\beta:=\e^{-l}>1$. Since the left-hand side of \eqref{eq qI12} does not depend on $\beta$, we can optimize the right-hand side over $l>0$, to find that the best bound in \eqref{eq qI12} is obtained when
      \begin{align}\label{eq l}
      l:=\frac{\alpha(\alpha-d+\gamma)}{(d+2+\alpha-\gamma)+\alpha(\alpha-d+\gamma)}.
      \end{align}
       This particular choice yields
                 \begin{align}\label{eq uZ}
I_1 \le c_{4,4} \e^{H+\frac{d+2-\alpha-\gamma}{2\alpha}l},
      \end{align}
 for some constant $c_{4,4}\in (0,\infty)$.

                     By Lemma   \ref{Lem  spatial} and using the same argument in the proof of Theorem 1.1 in \cite{FKM2015}, we obtain that  there exists  a   constant $c_{4,5}\in (0,\infty)$ such that
                     \begin{align}\label{eq l2}
      I_2\le c_{4,5} \e, \ \ \ \ \text{uniformly for all } \e>0.
      \end{align}
     The estimates      \eqref{eq uZ} and     \eqref{eq l2} together imply the desired result \eqref{eq q1}.

                    The proof is complete.
                       \end{proof}

     By using the arguments in \cite{FKM2015} and using the estimates in the proof of Theorem \ref{thm spatial gradient}, we can obtain
     Corollaries  \ref{cor  t LIL}-\ref{prop variations}, whose proofs are omitted here.

  \vskip0.5cm

\noindent{\bf Acknowledgments }   { The author is grateful to the anonymous referee for the corrections. He would like to thank  Professor Y. Xiao for his useful discussions.  This research  is supported by  NNSFC grants 11871382.}

\bigskip


\vskip0.5cm

\begin{thebibliography}{123}

   \bibitem{Appl} D.  Applebaum. {\it  L\'evy processes and stochastic calculus.} Cambridge university press, 2009.

  \bibitem{ANTV2022}
 O.  Assaad,   D. Nualart,  C. A. Tudor, L.   Viitasaari. Quantitative normal approximations for the stochastic fractional heat equation. {\it Stoch PDE: Anal. Comp.}, {\bf 10}, 223--254, 2022.

   \bibitem{Balan12} R. M. Balan.   Linear SPDEs driven by stationary random distributions. {\it J.  Fourier Anal. Appl.}, {\bf 18}(6), 1113--1145, 2012.

 
 \bibitem{BG1960}
 R. M. Blumenthal,  R. K. Getoor.   Some theorems on stable processes. {\it Trans. Am. Math. Soc.}, {\bf 95},
263--273,  1960.

 \bibitem{CZ2016}
 Z.-Q. Chen, X.   Zhang.  Heat kernels and analyticity of non-symmetric jump diffusion semigroups. {\it Probab. Theory Relat. Fields},  {\bf 165}(1), 267--312, 2016.

  \bibitem{Dalang1999}
R. C. Dalang.  Extending martingale measure stochastic integral with applications to spatially homogeneous S.P.D.E's. {\it  Electron. J. Probab.}, {\bf 4}, 1--29, 1999.

\bibitem{DF1998}
R. C. Dalang, N.E.  Frangos.  The stochastic wave equation in two spatial dimensions. {\it Ann. Probab.}, {\bf 26}(1), 187--212, 1998.

 \bibitem{DKN2013}
R. C. Dalang,  D. Khoshnevisan,    E. Nualart. Hitting probabilities for systems of non-linear stochastic heat equations in spatial dimension $ k\ge 1$. {\it Stoch. Partial Differ. Equ. Anal. Comput.}, {\bf1}(1), 94--151, 2013.


 \bibitem{DQ2011}
R. C. Dalang, L. Quer-Sardanyons.  Stochastic integrals for spde's: a comparison. {\it Expo. Math.},  {\bf 29}(1), 67--109, 2011.


\bibitem{Das2022}  { S. Das. Temporal increments of the KPZ equation with general initial data.  Preprint available at \href{https://arxiv.org/abs/2203.00666}{arXiv:2203.00666}, 2022.}



\bibitem{FK2013}
M.  Foondun, D. Khoshnevisan. On the stochastic heat equation with spatially-colored random forcing. {\it Trans. Am. Math. Soc.}, {\bf 365}, 409--458,  2013.

\bibitem{FKM2015}
M. Foondun,  D.  Khoshnevisan, P.   Mahboubi.   Analysis of the gradient of the solution to a stochastic heat equation via fractional brownian motion.  {\it Stoch. Partial Differ. Equ. Anal. Comput.}, {\bf 3}(2), 133--158, 2015.

 \bibitem{HSWX20}  {
R. Herrell, R. Song, D. Wu,  Y. Xiao. Sharp space-time regularity
of the solution to a stochastic heat equation driven by a fractional-colored noise.
{\it Stoch. Anal. Appl.}, {\bf 38}, 747--768, 2020.}

\bibitem{KT2019}
Z. M.  Khalil, C. A. Tudor.   On the distribution and $q$-variation of the solution to the heat equation with fractional Laplacian. {\it Probab.  Math. Statist.}, {\bf 39}(2), 315--335, 2019.

 \bibitem{KS2022} D.
Khoshnevisan,   M. Sanz-Sol\'e.  Optimal regularity of SPDEs with additive noise.   arXiv:2208.01728, 2022.


\bibitem{KSXZ2013} {
 D.  Khoshnevisan, J. Swanson,  Y. Xiao,  L.  Zhang. Weak existence of a solution to
a differential equation driven by a very rough fBm.  Preprint available at
 \href{https://arxiv.org/pdf/1309.3613.pdf}{arXiv:1309.3613v2}, 2014.}

\bibitem{LN2009}
P. Lei,  D.  Nualart.  A decomposition of the bifractional Brownian motion and some applications. {\it Statist.   Probab. Lett.}, {\bf 79}(5), 619--624, 2009.



\bibitem{MT02} {
C. Mueller, R. Tribe.  Hitting probabilities of a random string. {\it Electronic J. Probab.} {\bf 7}, Paper No. 10, 1--29, 2002.}


\bibitem{SW2020}
S.  Shang,  R.   Wang. Transportation inequalities under uniform metric for a stochastic heat equation driven by time-white and space-colored noise.  {\it Acta Appl. Math.},  {\bf170}, 81--97, 2020.


\bibitem{TX17} {
C. A. Tudor, Y. Xiao.  Sample path properties of the solution to the fractional-colored stochastic heat equation. {\it Stoch. Dyn.},  {\bf 17}(1), 1750004, 2017.}


 \bibitem{Walsh1986}
 J. Walsh. {\it   An introduction to stochastic partial differential equations}, \`Ecole d'\'Et\'ede Probabilit\'es de Saint-Flour, XIV--1984, Lecture Notes in Math. 1180, Springer, Berlin,   265--439, 1986.

 \bibitem{WX2023} {
 R.   Wang, Y. Xiao. Temporal properties of the stochastic  fractional heat equation with spatially-colored noise. Preprint, 2023.}


\end{thebibliography}
\end{document}